\documentclass[11pt,reqno]{amsart}
\usepackage{amsmath}
\usepackage[latin1]{inputenc}
\usepackage [english] {babel}
\usepackage [T1] {font enc}
\usepackage{amsthm}
\usepackage{textcomp}
\usepackage{enumerate}
\usepackage{indentfirst}
\usepackage{color}
\usepackage{graphicx}
\usepackage{newlfont}
\usepackage{amssymb}
\usepackage{amsmath}
\usepackage{latexsym}
\usepackage{mathrsfs}
\usepackage{fancyhdr}
\usepackage{amsfonts}
\usepackage{graphicx}
\usepackage{pdfsync}
\usepackage{color}

\newtheorem{theorem}{Theorem}[section]
\newtheorem{lemma}[theorem]{Lemma}

\newtheorem{remark}[theorem]{Remark}

\newtheorem{proposition}[theorem]{Proposition}
\numberwithin{equation}{section}

\renewcommand{\epsilon}{\varepsilon}
\renewcommand{\theta}{\vartheta}
\renewcommand{\rho}{\varrho}
\renewcommand{\phi}{\varphi}
\AtBeginDocument{}

\newcommand{\numberset}{\mathbb}
\newcommand{\N}{\numberset{N}}
\newcommand{\R}{\numberset{R}}
\newcommand{\Z}{\numberset{Z}}

\newcommand{\LM}[1]{\hbox{\vrule width.2pt \vbox to#1pt{\vfill \hrule width#1pt height.2pt}}}
\newcommand{\LL}{{\mathchoice
{\,\LM7\,}{\,\LM7\,}{\,\LM5\,}{\,\LM{3.35}\,}}}

{\left\lbrace\begin{array}{@{}l@{}}}%
{\end{array}\right.}

\DeclareMathAlphabet{\mathpzc}{OT1}{pzc}{m}{it}
\newcommand{\cstar}{\stackrel{*}{\rightharpoonup}}
\newcommand{\h}{\mathcal{H}^{1}}
\newcommand{\mline}{b\otimes t\h \mrestr (t\R)}
\newcommand{\M}{\mathcal{M}_{\mathrm{df}}^{(m)}}
\newcommand{\Rc}{\mathcal{R}_{1}^{\mathrm{cl}}(\Omega,\Z^{m})}

\newcommand{\Sa}{\mathcal{S}}
\newcommand{\xo}{x_{0}}

\newcommand{\Tj}{\widehat{T}_{j}}
\newcommand{\Mb}{\mathbf{M}}

\newcommand{\su}{\mathrm{supp}}

\newcommand{\sbb}{\subset\subset}
\newcommand{\lt}{\Vert T\Vert}

\newcommand{\Fl}{\mathbf{F}}
\newcommand{\rz}{\mathcal{R}_{1}(\Omega,\Z^{m})}
\newcommand{\ltl}{\lvert\theta\lvert}

\newcommand{\bp}{\bar{\psi}}

\newcommand{\hy}[1]{{\operatorname{\mathit{#1}}}}

\newcommand{\mrestr}{\LL}

\begin{document}
\title[Homogenization of energies defined on $1$-rectifiable currents] {Homogenization of energies defined on $1$-rectifiable currents}

\author[A. Garroni]
{A. Garroni}
\address[Adriana Garroni]{Dipartimento di Matematica ``Guido Castelnuovo'', Sapienza Universit\`a di Roma, P.le Aldo Moro 5, I-00185 Roma, Italy}
\email[A. Garroni]{garroni@mat.uniroma1.it}

\author[P. Vermicelli]
{P. Vermicelli}
\address[Pietro Vermicelli]{}
\email[P. Vermicelli]{pietro@pietrovermicelli.com}

\begin{abstract}
In this paper we study the  homogenization of a class of energies concentrated on lines. In dimension $2$ (i.e., in codimension $1$) the problem reduces to the homogenization of partition energies studied by \cite{AB}. There, the key tool is the representation of partitions in terms of  $BV$ functions with values in a discrete set. In our general case the key ingredient is the representation of closed loops with discrete multiplicity either as divergence-free matrix-valued measures supported on curves or with $1$-currents with multiplicity in a lattice. In the $3$ dimensional case the main motivation for the analysis of this class of energies is the study of line defects in crystals, the so called dislocations.
\end{abstract}

\maketitle

\begin{center}
{\em Dedicated to Umberto Mosco in the occasion of his 80th birthday}
\end{center}

\section{Introduction}

In the present work we consider energies concentrated on lines of the form
\begin{equation}\label{funzomoge}
\int_{\Omega\cap\gamma}\psi(y,\theta(y),\tau(y))\,d\h(y)\,,
\end{equation} 
where $\gamma$ is a  $1$-rectifiable set in $\Omega\subset\R^n$ given by the union of closed loops, the function $\theta:\gamma\rightarrow\Z^{m}$ is a vector-valued multiplicity, constant on each closed loop of $\gamma$, and  $\tau:\gamma\rightarrow\Sa^{n-1}$ is the tangent vector defined $\h$-a.e.~on $\gamma$. 
The main result of the paper concerns the homogenization of energies of this class. This can be expressed as the characterization of the limit of scaled energies 
\begin{equation}\label{funzomogeps}
\int_{\Omega\cap\gamma}\psi\Bigl({x\over\epsilon},\theta(x),\tau(x)\Bigr)\,d\h(x)\,,
\end{equation} as $\epsilon\to 0$ when $\psi$ is periodic in the first variable.

The main motivation for the study of the energies above comes from the analysis of dislocations in crystals. Dislocations are defects in the crystalline structure of metals that are crucial for the understanding of plastic behaviours. 
At a continuum level they can be interpreted as line singularities carrying an energy of the form above, where the multiplicity $\theta$ is the so-called Burgers vector which belongs to a lattice (which is a material property), that we can assume to be $\Z^m$. The function $\psi$ represents the line tension energy density that can be computed \`a la Volterra (see e.g.  \cite{CGO}) using continuum elasticity.

The asymptotic behaviour of energies of the form \eqref{funzomogeps} can be expressed in terms the computation of their $\Gamma$-limit with respect to a suitable convergence. In a two-dimensional setting this has been carried over in \cite{AB} in a BV setting. Indeed, in that case the system of lines can be interpreted as the set of the interfaces of a Caccioppoli partition, or, equivalently, the set of discontinuity points of a BV-function taking values in a discrete set. In that functional setting energies of the form above can be analysed as $\Gamma$-limits with respect to the BV-convergence. 
In higher dimension, instead of identifying systems of loops with partitions, they can be interpreted as divergence-free measures or $1$-rectifiable currents without boundary. Throughout the paper we will make use of those two standpoints interchangeably, taking advantage of the possibility of choosing the most suited of the two in technical points. The corresponding equivalent notions of convergence makes it possible to study energies \eqref{funzomogeps} in terms of $\Gamma$-convergence.

The main result of the paper is that, under suitable growth assumptions of $\psi$ the $\Gamma$-limit of energies \eqref{funzomogeps} as $\epsilon\to0$ exists and can be written as
\begin{equation}\label{funzomoge}
\int_{\Omega\cap\gamma}\psi_{\mathrm{hom}}(\theta(y),\tau(y))\,d\h(y)\,,
\end{equation} 
for a suitable function $\psi_{\mathrm{hom}}$. Furthermore, this function can be characterized by an asymptotic formula.

In order to prove that the function given by the asymptotic homogenization formula gives a lower bound for the $\Gamma$-limit, following \cite{Bra}, we make use of Fonseca and M\"uller blow-up technique \cite{FM}. This method, originally introduced to deal with relaxation problems, works nicely with homogenization problems as well. Here, as in \cite{CGM}, it will be useful to rephrase the problem in terms of closed 1-rectifiable currents, as this allows an easy treatment of the possibility of fixing boundary conditions, which is a technical point necessary to carry over the blow-up method. In order to prove the upper bound we proceed by density using the homogenization formula  explicit construction of recovery sequences. We will need some results on divergence-free measures for which we refer to \cite{CGM} instead.

\section{Formulation of the problem}

Let $\Omega\subset\R^{n}$ be an open set with Lipschitz boundary. Let  $\psi:\R^{n}\times\Z^{m}\times\Sa^{n-1}\rightarrow[0,\infty)$ be the energy density of the energy \eqref{funzomoge}. We assume that $\psi$ ia a Borel function and satisfies
\begin{equation}\label{growth}
c_{0}\ltl\leq\psi(y,\theta,\tau)\leq c_{1}\lvert\theta\lvert \qquad c_{0},c_{1}>0,\,\,\forall y\in\mathbb{R}^{n},\,\,\forall\tau\in\mathcal{S}^{n-1}, \ \ \forall\vartheta\in\mathbb Z^n.
\end{equation}

A convenient framework to represent the set of admissible configurations is the one of divergence-free matrix-valued measures  or alternatively of 1-rectifiable currents without boundary.
\smallskip

{\em Representation with measures:}

Following \cite{CGM}, we will denote by $\M(\Omega)$ the set of divergence-free (in the sense of distribution) measures $\mu\in\mathcal{M}(\Omega;\R^{m\times n})$ of the form 
$$
\mu=\theta\otimes\tau\h\LL\gamma,
$$ 
where $\theta:\gamma\rightarrow\Z^{m}$ is a $\h\mrestr\gamma$-integrable function, $\gamma$ a 1-rectifiable set and $\tau:\gamma\rightarrow\Sa^{n-1}$ its tangent vector defined $\h$-a.e. on $\gamma$. 
The divergence-free conditions reads as
\begin{equation}\label{div-free}
\int_\gamma \theta\cdot(D\phi \tau) d\mathcal{H}^1=0,
\end{equation}
for all $\phi\in C_0^\infty(\Omega;\mathbb{R}^m)$.

So that, for any $\mu\in\M(\Omega)$ the energy in \eqref{funzomoge} is denoted by 
\begin{equation}\label{funzomoge1}
F(\mu)=\int_{\Omega\cap\gamma}\psi(y,\theta(y),\tau(y))\,d\h(y)\,.
\end{equation} 
In particular by the growth condition \eqref{growth}
we deduce that the energy 
$$
F(\mu)\geq c_0 |\mu|(\Omega), 
$$
and it is coercive with respect to the weak$^*$ convergence of measures. Indeed a sequence with bounded energy is in particular bounded in the total variation, and therefore it is compact in the weak$^*$ converge. 
The fact that $\M(\Omega)$ is closed with respect to the weak$^*$ convergence can be seen, as already mentioned, in a very efficient way by using the approach of geometric measure theory, i.e., the setting of rectifiable currents \`a la Federer and Fleming extended to the case of currents with vector multiplicity.

\smallskip
\smallskip

{\em Representation with integral $1$-rectifiable currents:}

We denote by $\mathcal{R}_{1}(\Omega,\Z^{m})$  the set of $\Z^{m}$-valued 1-rectifiable currents. Let $\theta$, $\gamma$, $\tau$ be as before, then $T\in\mathcal{R}_{1}(\Omega,\Z^{m})$ is a functional on the space of smooth compactly supported 1-forms that admits the following representation
\begin{equation}\label{def1current}
\langle T,\phi\rangle=\int_{\gamma}\theta\langle\phi;\tau\rangle d\h\in\R^{m},\quad\forall\phi\in C^{\infty}_{c}(\Omega,\R^{n}).
\end{equation} 
We recall that the boundary of a 1-rectifiable current $T$ is the 0-current $\langle\partial T,\phi\rangle=\langle T,d\phi\rangle$ for all $\phi\in C^{\infty}_{c}(\Omega)$; finally a current is \textit{closed} or \textit{without boundary} if $\partial T=0$. We denote by $\Rc$ the set of currents in $\mathcal{R}_{1}(\Omega,\Z^{m})$ which are closed.

Now there is a $1$-to-$1$ correspondence between measures in $\M(\Omega)$ and currents in $\mathcal{R}_{1}(\Omega,\Z^{m})$ with no boundary.  Indeed it is immediate to see that the divergence free condition \eqref{div-free} translates in the condition of having zero boundary for the corresponding currents.
Therefore for any $\mu\in\M(\Omega)$ we denote by $T(\mu)$ the corresponding current in $\Rc$ and for any $T\in \Rc$ we denote by $\mu(T)$ the corresponding measure in $\M(\Omega)$.

In particular given $\mu\in \M(\Omega)$ the mass of the associated current $T(\mu)$ is given by
$$
\Mb(T(\mu))=|\mu|(\Omega).
$$
Moreover the weak$^*$ convergence of a sequence of measure $\mu_j$ translates exactly in the weak$^*$ convergence of the corresponding currents $T_j$, i.e.,
$$
\lim_j\,\langle T_j,\phi\rangle = \langle T,\phi\rangle \quad \forall \phi\in C_{c}(\Omega,\R^n).
$$

The advantage of using the language of rectifiable currents relies on a rich theory which guarantees a structure result which allows to characterize all currents in $\Rc$ as a countable family of Lipschitz closed loops with constant multiplicity, a compactness result for sequence with bounded mass in $\Rc$ (see \cite{FF}), and a good  approximation of currents (and therefore of the corresponding measures) with polyhedral currents (i.e. currents in $\Rc$ supported on polyhedral curves). These results in the formulation that is needed here are recalled in the Appendix.

Our results  will be mainly stated using the more familiar notation of measures, but throughout the paper, especially in the proofs, we will use the two notations interchangeably, making sure to highlight the advantages of one over the other.

\section{The homogenization theorem}

We now state the main result of this paper which concerns the homogenization of the energies concentrated on lines.

In the following $\Omega\subset\mathbb{R}^{n}$ is a bounded, open set with Lipschitz boundary and $\psi: \mathbb{R}^{n}\times\mathbb{Z}^{m}\times\mathcal{S}^{n-1}\rightarrow [0,\infty)$ is a Borel function satisfying \eqref{growth}.

Additionally we assume that $\psi$ is $1$-periodic in the first variable and we define
\begin{equation}\label{funzomo3}
F_{\epsilon}(\mu)=\int_{\Omega\cap\gamma}\psi\bigg(\frac{y}{\epsilon},\theta(y),\tau(y)\bigg)d\h(y)\qquad \mu\in\M(\Omega).
\end{equation}

The main result of this paper is the characterization of the $\Gamma$-limit of the functional $F_\epsilon$ as $\epsilon\to 0$.

We stress that a $\Gamma$-convergence result must be complemented by a compactness result: together they guarantee the convergence of minima. Here compactness in the class $\M(\Omega)$, as we noted, is a consequence of the divergence free constraint. This is the first point in which the already mentioned equivalence between the representation with measures and the one with 1-currents gives the its contribution.

\begin{theorem}[Compactness] \label{compact-th}Let $\mu_\epsilon\in\M(\Omega)$ be a sequence of measures satisfying 
	\begin{displaymath}
	F_{\epsilon}(\mu_\epsilon)\leq C
	\end{displaymath}
	for some $C>0$, then there is a measure  $\mu \in\M(\Omega)$ and a subsequence $(\mu_{\epsilon_{j}})$ such that
	\begin{displaymath}
	\mu_{\epsilon_{j}}\cstar \mu.
	\end{displaymath}
\end{theorem}

\begin{proof}
	The result is an immediate consequence of the lower bound for the energy density $\psi$ given by \eqref{growth}. Indeed $F_{\epsilon}(\mu_\epsilon)\leq C$ implies that 
	$$
	c_0 |\mu_\epsilon|(\Omega)= c_0 \int_{\Omega\cap\gamma_\epsilon}|\theta_\epsilon|(y)d\h(y)\leq C.
	$$
	
	Then, up to a subsequence the sequence of measure $\mu_\epsilon$ weakly$^*$ converge to some matrix valued Radon measure $\mu$. In order to conclude it is enough to show that the limit measure as the right structure, i.e., it is concentrated on a $1$-rectifiable set $\gamma$ and its density is of the form $\theta\otimes \tau$, in other words it belongs to $\M(\Omega)$.
	
	This is an immediate consequence of the result of compactness of $1$-rectifiable currents with integer multiplicity. Indeed the family of currents $T(\mu_\epsilon)$ satisfy
	$$
	\Mb(T(\mu_\epsilon))\le C\qquad \Mb(\partial T(\mu_\epsilon))=0
	$$
	and hence up to a subsequence it converges to a current in $T\in \Rc$. Then $\mu(T)\in \M(\Omega)$ and it is the weak$^*$ limit of the corresponding subsequence of $\mu_\epsilon$.
\end{proof}

\begin{remark}
This compactness result clarifies the importance of the divergence free condition for the measures $\mu$. It is indeed easy to construct a sequence of measures supported on $1$-rectifiable sets with integer multiplicities (equivalently a sequence of integer $1$-currents) which converges to a measure which is not supported on curves, for instance which converges to the $n$-dimensional Lebesgue measure. This can be done by taking a collection of many uniformly distributed short segments. In this example the corresponding current has large boundary and therefore the compactness result does not apply.

Moreover the above theorem sets the right topology in which we have to study the $\Gamma$-limit of the $F_\epsilon$.
\end{remark}

 Before stating the theorem, it is useful to introduce the following notation. For all $t\in\mathcal{S}^{n-1}$ choose a rotation $O_{t}\in\mathrm{SO}(n)$ with $O_{t}e_{1}=t$. Then, for every $h,l>0$ we define the rectangle 
 \begin{displaymath}
 R_{l,h}^{t}=O_{t}\bigg[\bigg[-\frac{l}{2},\frac{l}{2}\bigg]\times\bigg[-\frac{h}{2},\frac{h}{2}\bigg]^{n-1}\bigg]
 \end{displaymath}
 of height $h$ and a side $l$, centred at the origin, and one side parallel to the direction $t$. If the rectangle is centred in a point $x\in\R^n$ we denote it by $ R_{l,h}^{t}(x)$.
 Similarly we denote by $Q_{h}^{t}(x)$ the cube of side $h$, and one side parallel to the direction $t$, centred at $x\in\Omega$, i.e., $Q_{h}^{t}(x)=R_{h,h}^{t}(x)$. If $x=0$ we drop the $x$ and write $Q_{h}^{t}$.

\begin{theorem}[Homogenization]\label{thm1}
Assume that $\psi$, satisfying \eqref{growth}, is $1$-periodic in the first variable, then the functionals $F_{\varepsilon}$ in $($\ref{funzomo3}$)$ $\Gamma$-converge as $\epsilon\to 0$, with respect to the weak$^{\ast}$ convergence of measures, to the functional defined by
	\begin{displaymath}
	F_{\mathrm{hom}}(\mu)=\begin{cases}\displaystyle
	{\int_{\gamma\cap\Omega}\psi_{\mathrm{hom}}(\theta,\tau)d\mathcal{H}^{1}} & \mu\in\M(\Omega)\\ 
	+\infty & \hbox{otherwise},
	\end{cases}
	\end{displaymath}
	where
	for all $\textrm{b}\in\mathbb{Z}^{m}$ and $\textrm{t}\in\mathcal{S}^{n-1}$ the effective energy density is given by
	\begin{align}
	\label{cell-problem}
	\psi_{\mathrm{hom}}(b,t)=\lim_{T \to \infty}\frac{1}{T}\inf\Bigg\lbrace\int_{Q_{T}^{t}\cap\gamma}&\psi(y,\theta(y),\tau(y)) d \mathcal{H}^{1}(y):\mu\in\mathcal{M}_{\mathrm{df}}^{(m)}(Q_{T}^{t}),\\
	\nonumber
	& \mathrm{supp}(\mu - b \otimes t \mathcal{H}^{1}\mrestr(t \mathbb{R}\cap Q_{T}^{t}))\subset Q_{T}^{t}\Bigg\rbrace.
	\end{align}

\end{theorem}

\begin{proof}
	The proof of the lower bound is given in Subsection \ref{lb} and uses the characterization of the effective energy density by means of the asymptotic formula (studied in Section \ref{cell}). The proof of the upper bound is instead given in Section \ref{upper}.

\end{proof}

\subsection{The cell problem formula}\label{cell}

A key ingredient is the analysis of the cell problem formula in \eqref{cell-problem} which characterizes the effective energy. 

Here and in the rest of the paper it is convenient to introduce the localised functionals for every Borel subset $A\subseteq\R^{n}$ and a measure $\mu=\theta\otimes\tau \h\LL\gamma\in\M(A)$,
\begin{equation}\label{local}
F_{\epsilon}(\mu,A)=\int_{A\cap\gamma}\psi\left(\frac{y}{\epsilon},\theta,\tau\right)d\h.
\end{equation}
In the next proposition we prove that the energy density $\psi_{\mathrm{hom}}$ is well defined through an asymptotic formula. Moreover we show that it is rather flexible and thanks to the periodicity of $\psi$ is not sensitive to the translations. As a consequence we will get the continuity of $\psi_{\mathrm{\hom}}$.

\begin{proposition}[Homogenization formula]\label{prop1}
	Let $(x_{T})$ be a family of points in $\mathbb{R}^{n}$, with $T\in\R$ and $\psi$ as in Theorem \ref{thm1}. Then for all $b\in \mathbb{Z}^{m}$ and $t\in\mathcal{S}^{n-1}$ the limit 
	\begin{align*}
	\lim_{T \to \infty}\frac{1}{T}\inf\Bigg\lbrace\int_{Q_{T}^{t}(x_{T})\cap\gamma}&\psi(y,\theta,\tau) d \mathcal{H}^{1}:\mu=\theta\otimes\tau \h\LL\gamma\in \mathcal{M}_{\mathrm{df}}^{(m)}(Q_{T}^{t}(x_{T}))\\& \mathrm{supp}(\mu - b \otimes t \mathcal{H}^{1}\mrestr(x_T+t \mathbb{R}))\subset Q_{T}^{t}(x_{T})\Bigg\rbrace
	\end{align*}
	exists and it is independent of $(x_{T})$. Therefore it coincides with $\psi_{\hom}(b,t)$.
\end{proposition}
\begin{proof}
	Let $\bar{\gamma_{t}}(x_T)$ be the straight line parallel to $t$ passing through $x_T$, i.e., $\bar{\gamma_{t}}(x_{T})=x_T+t\mathbb{R}$ and denote $\bar{\gamma_{t}}=\bar{\gamma_{t}}(0)$. 
	Note that, in general, the line $\bar{\gamma_{t}}(x_{T})$ may not intersect the set of points in $\mathbb{Z}^{n}$. 
	
	Let $S\gg T$. Fix $L>n$ and consider the family of equispaced points $(z_{j})$ on $\bar{\gamma_{t}}(x_{S})$ with spacing $T+L$. For each point $z_{j}$ consider the point $y_{j}\in x_{T}+\mathbb{Z}^{n}$ such that 
	\begin{equation}\label{equispazio}
	\lvert y_{j} - z_{j}\lvert \leq \sqrt{n}.
	\end{equation}
	Let $\mu^{(T)}=\theta^{(T)}\otimes\tau^{(T)}\h\mrestr\gamma^{(T)}$ be a test measure for the minimum problem
	\begin{equation}\label{mt}
	m^{(T)}=\inf\Bigg\lbrace F_1(\mu,Q_{T}^{t}(x_{T})): \mu\in \M(\R^n)\ ,\, \ \mathrm{supp}(\mu - b \otimes t \mathcal{H}^{1}\mrestr\bar{\gamma_{t}}(x_{T}))\subset Q_{T}^{t}(x_{T})\Bigg\rbrace
	\end{equation} 
	such that
	\begin{equation}\label{equaminmu}
	F_1(\mu^{(T)},Q_{T}^{t}(x_{T}))=\int_{Q_{T}^{t}(x_{T})\cap\gamma^{(T)}}\psi(y,\theta^{(T)},\tau^{(T)})d\h<m^{(T)}+\frac{1}{T}.
	\end{equation}

	\begin{figure}\label{fig1}
		\centering
		\includegraphics[scale=0.7]{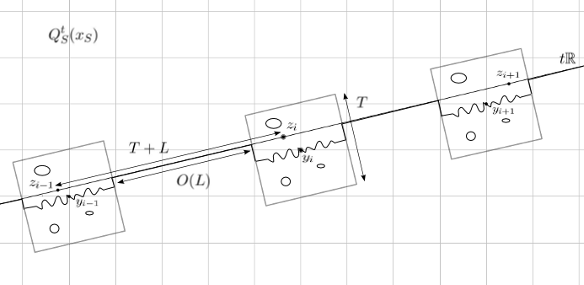}
		\caption{}
		\label{fig:pica22}
	\end{figure}

	Let $J$ be the set of indices such that $Q_{T}^{t}(y_{j})\subset Q_{S}^{t}(x_{S})$ and 
	\begin{displaymath}
	R_{S,T}=Q_{S}^{t}(x_{S})\setminus\bigcup_{j\in J}( Q_{T}^{t}(y_{j})).
	\end{displaymath}
	Hence define the function $h^j(x)= x-x_T+y_j$ and the measure
	\begin{displaymath}
	\widetilde{\mu}^{(S)}=\sum_{j\in J}h^j_\sharp(\mu^{(T)})\mrestr Q_{T}^{t}(y_{j})+b\otimes t\mathcal{H}^{1}\mrestr(\bar{\gamma_{t}}(x_{S})\cap R_{S,T}).
	\end{displaymath}  
Note that the measure $h_\sharp^j\mu^{(T)}$ is nothing else than the  measure obtained from $\mu^{(T)}$  translating the support from the cube $Q^t_T(x_T)$ to the cube $Q^t_T(y_j)$, i.e.,
$$
h_\sharp^j\mu^{(T)}= \theta^{(T)}(\cdot-x_T+y_j)\otimes\tau^{(T)}(\cdot-x_T+y_j)\h\mrestr(\gamma^{(T)}-x_T+y_j),
$$
therefore by the choice of $y_j$ we have
$$
	F_1(h_\sharp^j\mu^{(T)},Q_{T}^{t}(y_j))<m^{(T)}+\frac{1}{T}.
$$
	From $\widetilde{\mu}^{(S)}$, we can obtain a divergence free measure $\widehat{\mu}^{(S)}=\widehat{\theta}^{(S)}\otimes\widehat{\tau}^{(S)}\h\mrestr\widehat{\gamma}^{(S)}$ in $Q_{S}^{t}(x_{S})$, by connecting through a segment each endpoint of $\gamma^{(T)}+y_{i}-x_{T}$ on $y_{i}+Q_{T}^{t}$ to $\bar{\gamma}_{t}(x_{S})$. In doing so we have obtained a test measure for $m^{(S)}$ (see Figure \ref{fig1}). The conclusion follows if we prove 
	\begin{displaymath}
	\limsup_{S \to \infty}\frac{1}{S}m^{(S)}\leq\liminf_{T \to \infty}\frac{1}{T}m^{(T)}.
	\end{displaymath}
	In fact, using $\widehat{\mu}^{(S)}$ we get
	\begin{align*}
	\frac{1}{S}& m^{(S)}\leq\frac{1}{S}\int_{Q_{S}^{t}(x_{S})\cap\widehat{\gamma}^{(S)}}\psi(y,\widehat{\theta}^{(S)},\widehat{\tau}^{(S)})d\mathcal{H}^{1}\\[6pt] &=\frac{1}{S}\sum_{j\in J}F_1(h_\sharp^j\mu^{(T)},Q_{T}^{t}(y_j))+ \frac{1}{S}\int_{R_{S,T}\cap \bar{\gamma_{t}}(x_{S})}\psi(y,b,t)d\mathcal{H}^{1}\\[6pt]&
	\leq\frac{1}{S}\bigg[\frac{S}{T+L}\bigg]\bigg(m^{(T)}+\frac{1}{T}\bigg)+\frac{1}{S}\mathcal{H}^{1}(R_{S,T}\cap \bar{\gamma_{t}}(x_{S}))c+2C\sqrt{n}\frac{1}{S}\bigg[\frac{S}{T+L}\bigg]\\[6pt] &\leq \frac{1}{S}\bigg[\frac{S}{T+L}\bigg]\bigg(m^{(T)}+\frac{1}{T}\bigg)+\frac{1}{S}\bigg[S-T\bigg(\bigg\lvert\bigg[\frac{S}{T+L}\bigg]-2\bigg\lvert\bigg)\bigg]c\\[6pt]&+2C\sqrt{n}\frac{1}{S}\bigg[\frac{S}{T+L}\bigg]
	\end{align*}
	where we have used the choice of $\mu^{(T)}$ as in (\ref{equaminmu}) and  (\ref{equispazio}) to control the contribution of the segments in the measure $\widehat{\mu}^{(S)}$. The conclusion follows taking the $\limsup_{S \to\infty}$ first and then the $\liminf_{T \to \infty}$.
\end{proof}

\begin{remark}\label{uniform-cell}
Once the existence of the limit is proved it is also easy to see that this is not only independent of the choice of $x_T$ but also uniform in $x_T$. Indeed, by periodicity, we can assume that $x_T\in Q_1(0)$ which is compact, and then we deduce the uniformity.
\end{remark}
An important consequence of the above characterization of $\psi_{\mathrm{hom}}$ is its continuity.
\begin{proposition}[Continuity of $\psi_{\mathrm{hom}}$]\label{continuitapsihom} For all $b\in\Z^{m}$ and $t,t'\in\Sa^{n-1}$ let $\psi_{\hom}$ be as in Theorem \ref{thm1}. Then the following continuity property holds
	\begin{displaymath}
	\lvert \psi_{\mathrm{hom}}(b,t)-\psi_{\mathrm{hom}}(b,t')\lvert\leq c\lvert b\lvert\lvert t-t'\lvert,
	\end{displaymath}
	with $c>0$ a constant that depends only on $\psi_{\mathrm{hom}}$.
\end{proposition} 
\begin{proof}
	We prove this property by means of the asymptotic formula. Suppose that $\psi_{\mathrm{hom}}(b,t')\geq\psi_{\mathrm{hom}}(b,t)$. Fix $\epsilon>0$ and let $\mu_{t}^{(T)}$ be a test measure for the minimum problem $m_{t}^{(T)}$ in  (\ref{mt}) (Proposition 1.12, with $x_{T}=0$), such that
	\begin{equation}\label{formulaminimoinzero}
	\int_{Q_{T}^{t}\cap\gamma_{t}^{(T)}}\psi(y,\theta_{t}^{(T)},\tau_{t}^{(T)})d\h< m_{t}^{(T)}+\epsilon
	\end{equation}
	From $\mu_{t}^{(T)}$, we can obtain a test measure $\widehat{\mu}_{t'}^{(T)}$ for the problem 
	\begin{equation*}
	\inf\bigg\{ \int_{Q_{T+T\lvert t-t'\lvert}^{t'}\cap\gamma}\psi(y,\theta,\tau)d\h:\mu\in\M(\R^n),\ \mathrm{supp}(\mu-b\otimes t'\h\mrestr(t'\R))\subset Q_{T+T\lvert t-t'\lvert}^{t'}\bigg\}
	\end{equation*}
	proceeding as in Figure \ref{figura45}  (with support on the bold line), adding two segments to obtain the divergence free condition in $Q_{T+T\lvert t-t'\lvert}^{t'}\supset Q_{T}^{t}$ for the measure $\widehat{\mu}_{t'}^{(T)}$ and to get the right boundary condition.
	\begin{figure}[t] 
		\centering 
		\def\svgwidth{8cm}
		\scalebox{0.8}{\includegraphics{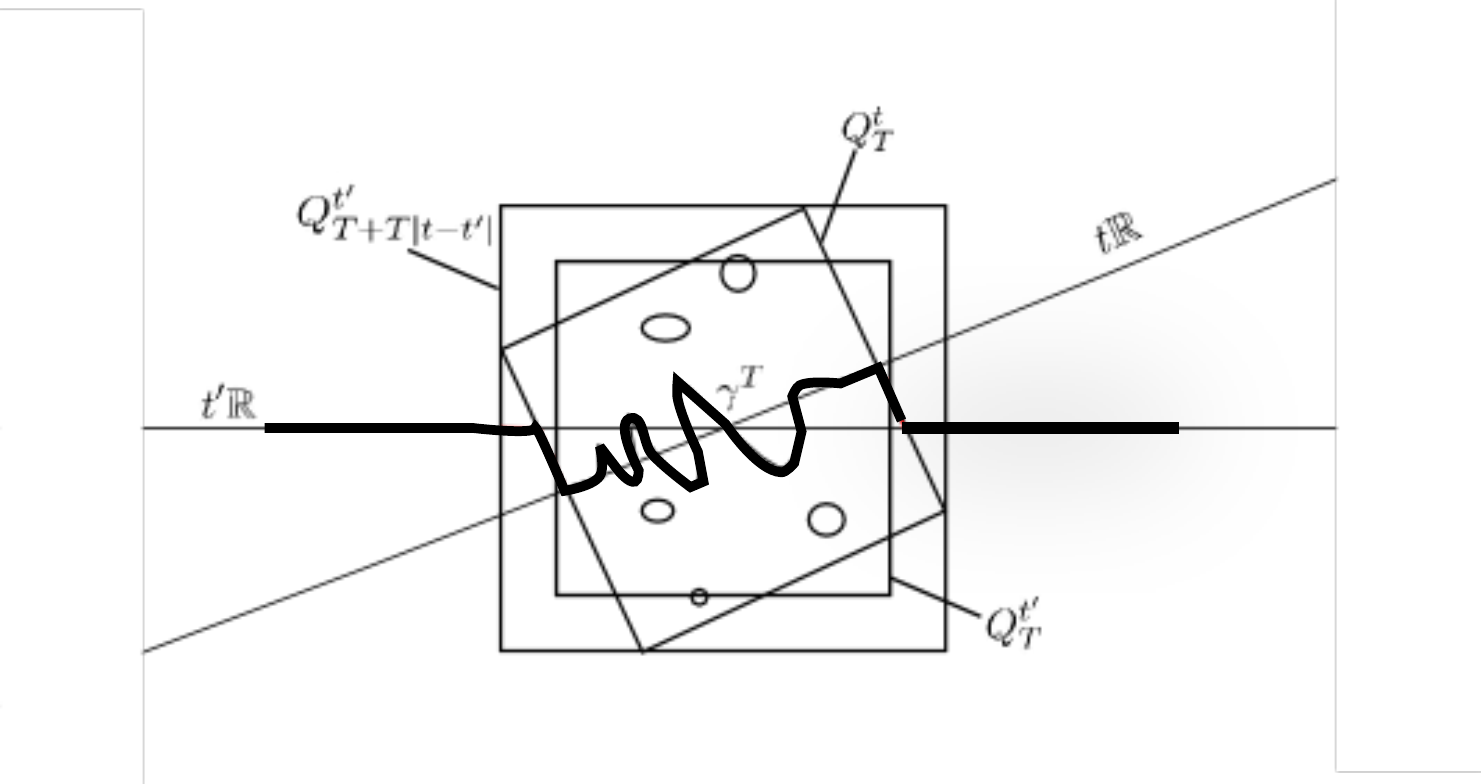}}
		\caption{The construction of the measure $\widehat{\mu}_{t'}^{(T)}$, for simplicity $t'=e_{1}$.}
		\label{figura45} 
	\end{figure}
	Then it is easy to see that 
	\begin{align*}
	\psi_{\mathrm{hom}}(b,t')-\frac{1}{T}m_{t}^{(T)}\leq\frac{1}{T}\bigg(&\int_{Q_{T+T\lvert t-t'\lvert}^{t'}\cap \widehat{\gamma}_{t'}^{(T)}}\psi(y,\widehat{\theta}_{t'}^{(T)},\widehat{\tau}_{t'}^{(T)})d\h \\& -\int_{Q_{T}^{t}\cap\gamma_{t}^{(T)}}\psi(y,\theta_{t}^{(T)},\tau_{t}^{(T)})d\h\bigg)+\frac{\epsilon}{T}
	\end{align*}
	and using the estimate from above for $\psi$ and by construction of $\widehat{\mu}_{t}^{T}$ we have
	\begin{displaymath}
	\psi_{\mathrm{hom}}(b,t')-\frac{1}{T}m_{t}^{(T)}\leq\frac{1}{T}c\lvert b\lvert T\lvert t-t'\lvert+\frac{\epsilon}{T}.
	\end{displaymath}
	We conclude taking the limit $T\to\infty$ and by the arbitrariness of $\epsilon$. To obtain the inequality with the opposite sign is enough to swap the roles of $t$ and $t'$.
\end{proof}

Finally we show that the asymptotic formula still holds if we replace the boundary condition with an approximate boundary condition. This point is essential in the proof of the lower bound.

To this aim we need the following technical lemma which allows one to modify the boundary value of a converging sequence of measures with a small change in the corresponding line energy.

\begin{lemma}\label{lem2}
	Let $\mu_{j}\in \M(Q_{1}^{t})$ be a sequence of divergence free measures in $Q_{1}^{t}$ such that
	\begin{equation}\label{equilimitatezzamasse1}
	\sup_{j}\lvert\mu_{j}\lvert(Q_{1}^{t})\leq M<\infty
	\end{equation}
	and $\mu_{j}\cstar\mu=b\otimes t\h\mrestr(t\R\cap Q_{1}^{t})\in\M(Q_{1}^{t})$. Then, $\forall\epsilon>0$, $\forall\delta_{0}\in(0,1/2)$ there exist a sequence $\delta_{j}$, $\delta_{0}/2<\delta_{j}<\delta_{0}$ and a sequence  $\widetilde{\mu}_{j}\in\M(Q_{1}^{t})$ such that $\widetilde{\mu}_{j}=\mu_{j}\mrestr (Q_{1-\delta_{j}}^{t})^{\circ}+b\otimes t\h\mrestr(Q_{1}^{t}\setminus \bar{Q}_{1-\delta_{j}}^{t})+\nu_{j}$, where $\nu_{j}\in [\mathcal{M}(Q_{1}^{t})]^{m\times n}$ is such that $\mathrm{supp}(\nu_{j})\subset\partial Q_{1-\delta_{j}}$ and $\lvert\nu_{j}\lvert(Q_{1}^{t})<\epsilon$ for every $j$ large enough (by $(Q_{1-\delta_{j}}^{t})^{\circ}$ we have denoted the interior of the cube $(Q_{1-\delta_{j}}^{t})$).
\end{lemma}

\begin{proof} 

	Fix $\epsilon\in(0,1)$ and $\delta_{0}\in (0,1/2)$ two independent parameters and consider $d$, the distance function from $\partial Q_{1}^{t}$. We now slice the measure $\widehat{\mu}_{j}=\mu_{j}-\mu_{b,t}$, with $\mu_{b,t}=b\otimes t\h\mrestr(t\R)$, through the function $d$, Lipschitz continuous with $\mathrm{Lip}(d)=1$. 
	
	Let $\widehat{\gamma}_{j}$ and $\widehat{\theta}_{j}$ be the support and multiplicity of $\widehat{\mu}_{j}$, respectively. Now the idea is to exploit the weak$^*$ convergence of $\widehat{\mu}_{j}$ to zero in order to find a $\delta_j$ for which $\mathrm{supp}\,\widehat{\mu}_{j}\cap \partial Q_{1-\delta_j }$ is the boundary of weighted segments with small mass; this will allow us to construct the measure $\nu_j$ in the statement. 
	Even though this can done by hands using the measures in $\M(\Omega)$ it is convenient also in this case to use the tools of currents which give a more general argument not confined to the case of $1$ dimensional objects.
	
	We denote $\Tj=T(\widehat{\mu}_{j})$ the current associated to $\widehat{\mu}_{j}$ and for every $\delta\in [0,\delta_{0}]$ we consider the current obtained by slicing $\Tj$ along the level sets of the function $d$ which we denote by
	\begin{displaymath}
	\Tj[d,\delta+]=(\partial\Tj)\mrestr\{x:d(x)>\delta\}-\partial(\Tj\mrestr\{x:d(x)>\delta\}).
	\end{displaymath}
	Since $\partial\Tj=0$ in $Q_{1}^{t}$, we have 
	$\Tj[d,\delta+]=-\partial(\Tj\mrestr\{x:d(x)>\delta\})$, i.e.
	$$
	\langle \Tj[d,\delta+], \phi \rangle =\int_{\gamma_j\cap \{d>\delta\}} \theta_j \langle \nabla\varphi,\tau_j\rangle d\h
	$$
	for all $\varphi\in C^\infty(Q_1^t)$.
	For slicing of currents it holds 
	\begin{equation}\label{eqM}
	\int_{\frac{\delta_{0}}{2}}^{\delta_{0}}\Mb(\Tj[d,s+])\, ds\leq \mathrm{Lip}(d)\Mb_{Q_{1-\delta_{0}/2}^{t}\setminus Q_{1-\delta_{0}}^{t}}(\Tj)<\Mb_{Q_{1-\delta_{0}/2}^{t}}(\Tj)
	\end{equation}
	and 
	\begin{equation}\label{eqF}
	\int_{\frac{\delta_{0}}{2}}^{\delta_{0}}\mathbf{F}_{\bar{Q}_{1-\delta_{0}/2}^{t}\cap\{d(x)=s\}}(\Tj[d,s+])\,ds\leq\mathrm{Lip}(d)\mathbf{F}_{\bar{Q}_{1-\delta_{0}/2}^{t}}(\widehat{T}_{j}),
	\end{equation}
where $\mathbf{F}_Q(T)$ denote the flat norm of  the current $T$	(see the Appendix).
	Since, by  \eqref{equilimitatezzamasse1}, the sequence $\Tj$ is equi-bounded in mass, from (\ref{eqM}) we get that, for  a.e. $\delta\in(\delta_{0}/2,\delta_{0})$, $\Mb(\Tj[d,\delta+])$ is finite. Hence, since $\Tj[d,\delta+]$ is a 0-rectifiable current for almost every $\delta$,  we have
	\begin{equation}\label{eqM1}
	\Mb(\Tj[d,s+])=\sum_{x\in\partial Q_{1-\delta}\cap\widehat{\gamma}_{j}}\lvert\widehat{\theta}_{j}(x)\lvert, \quad \lvert\widehat{\theta}_{j}(x)\lvert\geq 1
	\end{equation} 
	where the sum runs through a finite number of points $x_{1},\ldots,x_{M}$,  with multiplicity $\widehat{\theta}_{j}(x_{i})$ and positive orientation if $\widehat{\gamma}_{j}$ exits from $Q_{1-\delta}$ at $x_{i}$ (these oriented points together with their multiplicity give  the boundary of $\widehat{T}_{j}\mrestr Q_{1-\delta}$).
	
	Moreover by \eqref{equilimitatezzamasse1}, we can exploit the equivalence between weak$^*$ convergence and convergence in the flat norm \cite[Theorem 31.2]{Sim}, and from the fact that $\widehat{T}_{j}\cstar 0$ to obtain the convergence to zero of the flat norm. Therefore $\forall\sigma<\frac{\delta_{0}}{4}\epsilon$ we can find $J_{\sigma}\in\N$ such that $\mathbf{F}_{\bar{Q}_{1-\delta_{0}/2}^{t}}(\widehat{T}_{j})<\sigma$ for all $j\geq J_{\sigma}$.	Thus, by (\ref{eqF}), for each $j\geq J_{\sigma}$, the sets
	\begin{equation}\label{aj}
	A_{j}=\bigg \{\delta\in(\delta_{0}/2,\delta_{0}): \mathbf{F}_{\bar{Q}_{1-\delta_{0}/2}^{t}\cap\{d(x)=\delta\}}(\Tj[d,\delta+])<\frac{4}{\delta_{0}}\sigma\bigg\}
	\end{equation}
	have positive $1$-dimensional Lebesgue measure. In fact 
	\begin{align*}
|A_{j}^{c}|&\leq\frac{\delta_{0}}{4\sigma}\int_{\frac{\delta_{0}}{2}}^{\delta_{0}}\mathbf{F}_{\bar{Q}_{1-\delta_{0}/2}^{t}\cap\{d(x)=s\}}(\Tj[d,s+])\,ds\\&\leq\frac{\delta_{0}}{4\sigma}\mathbf{F}_{Q_{1-\delta_{0}/2}^{t}}(\Tj)<\frac{\delta_{0}}{4},
	\end{align*} 
	hence
	\begin{displaymath}
|A_{j}|=\frac{\delta_{0}}{2}-|A_{j}^{c}|>\frac{\delta_{0}}{4}>0.
	\end{displaymath}
	Now, for each $j\geq J_{\sigma}$ choose a $\delta_{j}$ such that $\delta_{j}\in A_{j}$ and the sum (\ref{eqM1}) runs through a finite number of points.
	
	To get to the conclusion, we first show that the following minimum problems, well defined by our choice of $\delta_{j}$, have solution. By definition, for every $j\geq J_{\sigma}$,
	\begin{align*}
&	\mathbf{F}_{\bar{Q}_{1-\delta{0}/2}^{t}\cap\{d(x) =\delta_{j}\}}(\Tj[d,\delta_{j}+])\\
&\quad=\inf\{\Mb(R_{j})+\Mb(B_{j}): \mathrm{supp}(R_{j}),\,\mathrm{supp}(B_{j})\subset\partial Q_{1-\delta_{j}},R_{j}+\partial B_{j}=\Tj[d,\delta_{j}+]\}.	\end{align*}
	Since $R_{j}$ is a $\Z^{m}$-valued $0$-rectifiable current, then $\Mb(R_{j})>1$ and from
	\begin{displaymath}
	\mathbf{F}_{\bar{Q}_{1-\delta{0}/2}^{t}\cap\{d(x)=\delta_{j}\}}(\Tj[d,\delta_{j}+])<\epsilon< 1
	\end{displaymath}  
	we can conclude that $R_{j}=0$. It follows that, for every $j\geq J_{\sigma}$,
	\begin{align*}
&	\mathbf{F}_{\bar{Q}_{1-\delta{0}/2}^{t}\cap\{d(x)
=\delta_{j}\}}(\Tj[d,\delta_{j}+])\\
	&\quad=\inf \{\mathbf{M}(B_{j}):\  B_{j}\in\mathcal{R}_{1}(Q_{1}^{t};\Z^{m}),\mathrm{supp}(B_{j})\subset\partial Q_{1-\delta_{j}},\,\partial  B_{j}=\Tj[d,\delta_{j}+]\}
	\end{align*}
	It's easily checked that, for all $j\geq J_{\sigma}$, the set over which we take the above infimum is not empty.
	Hence $\forall j\geq J_{\sigma}$, by means of the direct method, there exists a $1$-rectifiable current $S_{j}\in \mathcal{R}_{1}(Q_{1}^{t},\Z^{m})$ that satisfies the following properties
	\begin{enumerate}[(i)]
		\item $\mathrm{supp}(S_{j})\subset\partial Q_{1-\delta_{j}}$;
		\item $\Mb(S_{j})=\mathbf{F}_{\bar{Q}_{1-\delta_{0}/2}^{t}\cap\{d(x)=\delta_{j}\}}(\Tj[d,\delta_{j}+])<\frac{4}{\delta_{0}}\sigma <\epsilon$;
		\item $\partial S_{j}=\Tj[d,\delta_{j}+]$. 
	\end{enumerate}
	Now  
	we conclude defining $\nu_j= \mu(S_j)$  for $j\geq J_{\sigma}$  and  equal to $\nu_j = b\otimes t\h\mrestr(Q_{1-\delta_{j}}^{t})^{\circ}-\mu_{j}\mrestr (Q_{1-\delta_{j}}^{t})^{\circ}$ otherwise.

\end{proof}

\begin{remark}\label{rklem2}  By a scaling argument the result in Lemma \ref{lem2} still hold in the domain $Q_T^t$. 
	
	Moreover by a diagonal argument  under the assumptions of Lemma \ref{lem2} one could show that, for every $\delta$, 
there exists 
a sequence  $\widetilde{\mu}_{j}\in\M(Q_{T}^{t})$ such that $\widetilde{\mu}_{j}\cstar b\otimes t\h\mrestr Q_{T}^{t}$, with   $\mathrm{supp}(\widetilde{\mu}_{j}- b\otimes t\h\mrestr Q_{T}^{t}+\nu_{j})\subset Q^t_{T(1-\delta)}$. Moreover, exploiting the construction of $\widetilde{\mu}_{j}$ and its decomposition in mutually singular measures we also deduce that
\begin{equation}\label{rklem2-eq}
\lim_j F_1(\widetilde{\mu}_{j}, Q_T^t)= \lim_j F_1(\mu_j, Q^t_{T(1-\delta)}) + F_1 (b\otimes t\h\mrestr Q_{T}^{t}, Q_T^t\setminus Q^t_{T(1-\delta) }).
\end{equation}
A further diagonal argument produces a sequence $\hat\mu_j$ such that $\hat{\mu}_{j}\cstar b\otimes t\h\mrestr Q_{T}^{t}$, with   $\mathrm{supp}(\hat{\mu}_{j}- b\otimes t\h\mrestr Q_{T}^{t}+\nu_{j})\subset\subset Q^t_{T}$ and
\begin{equation}\label{rklem2-eq2}
\lim_j F_1(\hat{\mu}_{j}, Q_T^t)\leq \lim_j F_1(\mu_j, Q^t_{T}). 
\end{equation}
	
\end{remark}

Using the above lemma we finally prove the following proposition which is essentially the $\Gamma$-convergence result when the limiting configuration is given by the measure $\mu_{b,t}:=b \otimes t \mathcal{H}^{1}\mrestr(t \mathbb{R})$, i.e., straight with constant multiplicity.

\begin{proposition}\label{prop-flat-limit}
Let $(y_{T})$ be a family of points in $\mathbb{R}^{n}$, with 
 $T\in\R$ and $\psi$ as in Theorem \ref{thm1}. For all $b\in \mathbb{Z}^{m}$ and $t\in\mathcal{S}^{n-1}$  we define
\begin{equation*}
\psi^*(b,t):=\inf\left\{\liminf_{L\to +\infty} F_\frac{1}{L}(\mu_L,Q_{1}^{t}(y_L)): \mu_L\in \M(\R^n) ,\, \ \mu_{L}- b \otimes t \mathcal{H}^{1}\mrestr(y_L+t \mathbb{R})\cstar 0\right\}\,.
 \end{equation*}
 Then
\begin{equation}\label{mt3}
 \psi_{\mathrm{\hom}}(b,t)=\psi^*(b,t).
\end{equation}

\end{proposition}

\begin{proof}

In order to show one inequality we construct a sequence $\mu_L$ admissible for the definition of $\psi^*(b,t)$. 

As above by the periodicity of the problem we can reduce to the case in which $y_L=0$. 
For any $T>0$ we consider a  family of equispaced points $(x_{i})$ with spacing $T$ on $t\R$. Given $\epsilon>0$, thanks to Remark \ref{uniform-cell}, we can find $T$ large enough such that we find  $\mu_i^{(T)}=\theta_i^{(T)}\otimes\tau_i^{(T)} d \h \LL \gamma_i^{(T)}$  a test measure for the minimum problem $m^{(T)}$ in (\ref{mt}) with $x_T=x_i$ such that
\begin{equation}\label{eqminmis}
\frac{1}{T}\int_{Q_{T}^{t}(x_i)\cap\gamma_i^{(T)}}\psi(y,\theta_i^{(T)},\tau_i^{(T)})\,d\h\leq \psi_{\mathrm{hom}}(b,t)+\epsilon.
\end{equation}
 In particular we have that $\mathrm{supp}(\mu_i^{(T)} - b \otimes t \mathcal{H}^{1}\mrestr(t \mathbb{R}))\subset Q_{T}^{t}(x_i)$. With this condition we can rescale and glue the measures $\mu_i^{(T)}$ and construct a sequence which is admissible for $\psi^*(b,t)$. Precisely, given $L>0$  we define the function $f^{1/L}(x)=\frac{x}{L}$ and denote $I_L=\{i:\ Q_{T/L}^t(x_i/L)\subset Q_{1}^{t}\}$,  $R^L= t\R\setminus \cup_{i\in I_T} Q_{T/L}^t(x_i/L)$, then the measure
\begin{displaymath}	
\mu_L=\sum_{i\in I_L} f_{\sharp}^{1/L}(\mu_i^{(T)})\LL Q_{T/L}^t(x_i/L) + b \otimes t \mathcal{H}^{1}\mrestr(t \mathbb{R}\cap R^L)
\end{displaymath}
is divergence free and satisfies $\mu_L\cstar b \otimes t \mathcal{H}^{1}\mrestr(t \mathbb{R})$ as $L\to\infty$. Therefore using the notation $\mu_L=\theta_L\otimes \tau_L \h\LL\gamma_L$ we have
\begin{align*}
F_\frac{1}{L}(\mu_L,Q_{1}^{t})=
&
\int_{Q_{1}^{t}\cap\gamma_L}\psi\big(Ly,\theta_L(y),\tau_L(y)\big)d\h(y)= 
\\&
= \sum_{i\in I_L}  F_\frac{1}{L}( f_{\sharp}^{1/L}(\mu_i^{(T)}),Q_{T/L}^t(x_i/L)) +\int_{R^{L}\cap Q_1^t}\psi(y,b,t)d\h(y)
\\&
=\sum_{i\in I_L}\frac1L F_1( \mu_i^{(T)},Q_{T}^t(x_i))+
\int_{R^{L}\cap Q_1^t}\psi(y,b,t)d\h(y)
\\&
\leq {\sharp(I_L)}\frac{T}{L} (\psi_{\mathrm{\hom}}(b,t)+\epsilon) +c_1 |b| \h(R^{L}\cap Q_1^t)
\\&
\leq \bigg[\frac{L}{T}\bigg ] \frac{T}{L}(\psi_{\mathrm{\hom}}(b,t)+\epsilon)+c_1 |b|\frac{T}L.
\end{align*} 
Thus
$$
\psi^*(b,t)\leq  \psi_{\mathrm{\hom}}(b,t)+\epsilon
$$
which gives one inequality.

In order to obtain the opposite inequality it is enough to observe that given $L$  and fixing  a sequence $\mu_L$ admissible for the minimum problem $\psi^*(b,t)$, by means  of Remark \ref{rklem2}  and \eqref{rklem2-eq2}, we find a sequence $\hat\mu_L$ satisfying  $\mathrm{supp}(\hat{\mu}_{L}- b\otimes t\h\mrestr t\R)\subset\subset Q^t_{1}$ such that
\begin{equation}\label{bordo1}
 \liminf_L F_{\frac1L}(\hat\mu_L, Q^t_{1})\leq\liminf_L F_{\frac1L}({\mu}_{L}, Q_1^t).
\end{equation}
Then  rescaling by $L$ the measure $\tilde\mu_L:=f_\sharp^{L}(\hat\mu_L)$  satisfies $\mathrm{supp}(\tilde\mu_{L}- b\otimes t\h\mrestr t\R)\subset\subset Q^t_{L}$, and
$$
m^{(L)}\leq \frac1L F_1(\tilde\mu_{L}, Q^t_{L})=F_{\frac1L}(\hat\mu_L,Q_1^t)
$$
which together with \eqref{bordo1} concludes the proof.
\end{proof}

\begin{remark}\label{upper-rem}
From the proof of the above proposition we also deduce that given $b\in \Z^m$ and $t\in \mathcal{S}^{n-1}$ for all $R>0$ there exists a sequence $\mu_j\cstar b\otimes t d\h\LL t\R$ such that ${\mathrm{supp}}(\mu_j-b\otimes t d\h\LL t\R)\subset Q^t_R(x_j)$ and
$$
\psi_{\hom}(b,t)R=\lim_{j\to\infty} F_{\epsilon_j}(\mu_j, Q_R(x_j)).
$$
\end{remark}

\subsection{Proof of the liminf inequality by blow-up}\label{lb}
\phantom{aa}

\par

We are ready to prove the Theorem \ref{thm1}. We start by proving the liminf inequality, i.e., we need to show that for all $\mu_{j}\in\M(\Omega)$, $\mu_{j}\cstar \mu$ and $\epsilon_{j}\to 0^{+}$ one has
\begin{displaymath}
F_{\mathrm{hom}}(\mu)\leq \liminf_{j\to\infty}F_{\epsilon_{j}}(\mu_{j}).
\end{displaymath} 

Let $\epsilon_{j}$, $\mu_{j}$ be as above; it's not restrictive to suppose that $\liminf_{j}F_{\epsilon_{j}}(\mu_{j})$ is finite and, by compactness (Theorem \ref{compact-th}), $\mu=\theta\otimes\tau\h\mrestr\gamma\in\M(\Omega)$. \\

We head to the conclusion in three steps.

{\em Step 1. (localization and decomposition)} 
For all $j\in\N$ we denote  $\nu_{j}$ the  positive measures
\begin{equation}\label{eqo}
\nu_{j}=\psi\bigg(\frac{y}{\epsilon_{j}},\theta_{j},\tau_{j}\bigg)\h\mrestr\gamma_{j}.
\end{equation}
which is the energy density of the 
 localised functional
\begin{equation}
F_{\epsilon_{j}}(\mu_{j},A)=\int_{A\cap\gamma_{j}}\psi\bigg(\frac{y}{\epsilon_{j}},\theta_{j},\tau_{j}\bigg)d\h.
\end{equation}

By the assumptions on $\psi$, the sequence $\nu_{j}$ is equi-bounded hence, by compactness, there exists a positive Radon measure $\nu$ on $\Omega$ such that, up to a subsequence, $\nu_{j}\cstar\nu$. Now consider the Radon-Nikodym decomposition of the measure $\nu$ with respect to the $1$-dimensional Hausdorff measure restricted to $\gamma$, i.e.,
\begin{equation} 
\nu=g\h\mrestr\gamma+\nu^{s}
\end{equation}
where $g=\frac{d\nu}{d\h\mrestr\gamma}$ is the density of the part of $\nu$ which is absolutely continuous with respect to $\h\mrestr\gamma$ and the singular part is denoted by $\nu^{s}$  (and it is a positive measure as well).

{\em Step 2. (definition of the blow-up)} Let $\xo\in \Omega\cap\gamma$ be a Lebesgue point for $g$ with respect to $\h\mrestr\gamma$. We can write
\begin{equation}\label{eq1}
g(\xo)=\lim_{\rho\to 0^{+}}\frac{\nu(Q_{\rho}^{t}(\xo))}{\h(Q_{\rho}^{t}(\xo)\cap\gamma)}=\lim_{\rho\to 0^{+}}\frac{\nu(Q_{\rho}^{t}(\xo))}{\rho}
\end{equation}
where $t=\tau(\xo)$ and the last equality holds for $\h$-a.e $x_{0}\in\gamma$ by the Besicovitch-Marstrand-Mattila Theorem.
The Besicovitch derivation theorem ensures that $\h$-a.e $x_{0}\in\Omega\cap\gamma$ is a Lebesgue point for $\nu$ with respect to $\h\mrestr\gamma$, i.e., a Lebesgue point for $g$; moreover we can also assume that $x_{0}$ is a Lebesgue point for $\theta$ and $\tau$. By definition of  $\h$-rectifiability and approximate tangent space and then we can assume that $\h$-a.e $\xo\in\gamma\cap\Omega$ satisfies the following property
\begin{equation}\label{eqa}
\theta(x_{0}+\rho y)\otimes\tau(x_{0}+\rho y)\h\mrestr\frac{\gamma-x_{0}}{\rho}\cstar\mline
\end{equation}
with $b=\theta(\xo)$, in every open, bounded subset of $\R^{n}$; in fact it follows from Theorem \ref{struttura} that $\gamma$ is the union of countably many closed Lipschitz curves on which $\theta$ is constant, hence $\theta$ is $\h\mrestr\gamma$-measurable and integrable. Since $\nu$ is finite, we have $\nu(\partial Q_{\rho}^{t}(\xo))=0$ up to at most countably many  $\rho>0$. For all such $\rho$ it holds
\begin{equation}\label{eq2}
\nu(Q_{\rho}^{t}(\xo))=\lim_{j\to\infty}\nu_{j}(Q_{\rho}^{t}(\xo)).
\end{equation}

Moreover for every $\rho>0$ we have
\begin{equation}\label{eqa2}
\theta_j(x_{0}+\rho y)\otimes\tau_j(x_{0}+\rho y)\h\mrestr\frac{\gamma_j-x_{0}}{\rho}\cstar\theta(x_{0}+\rho y)\otimes\tau(x_{0}+\rho y)\h\mrestr\frac{\gamma-x_{0}}{\rho}
\end{equation}
as $j\to +\infty$.

Then by a diagonalization argument on (\ref{eq1}), (\ref{eq2}), and \eqref{eqa2} we can extract a subsequence $\rho_{j}$ such that
\begin{equation}\label{eq3}
g(x_0)=\frac{d\nu}{d\h\mrestr\gamma}(\xo)=\lim_{j\to\infty}\frac{\nu_{j}(Q_{\rho_{j}}^{t}(\xo))}{\h(Q_{\rho_{j}}^{t}(\xo)\cap\gamma)}.
\end{equation} 
\begin{equation}\label{eqa4}
\theta_j(x_{0}+\rho_j y)\otimes\tau_j(x_{0}+\rho_j y)\h\mrestr\frac{\gamma_j-x_{0}}{\rho_j}\cstar\mline
\end{equation}
and $\frac{\rho_{j}}{\epsilon_{j}}\rightarrow\infty$ as $j\to\infty$.

{\em Step 3. (lower bound for the blow-up)}  Recalling the expression of $\nu_{j}$ and \eqref{eq1}, (\ref{eq3}) is equivalent to
\begin{equation}\label{eq4}
g(\xo)=\lim_{j\to\infty}\frac{1}{\rho_{j}}\int_{Q_{\rho_{j}}^{t}(\xo)\cap\gamma_{j}}\psi\bigg(\frac{y}{\epsilon_{j}},\theta_{j},\tau_{j}\bigg)d\h.
\end{equation}
By a change of variable we get
\begin{equation}\label{eq4bis}
\frac{1}{\rho_{j}}\int_{Q_{\rho_{j}}^{t}(\xo)\cap\gamma_{j}}\psi\bigg(\frac{y}{\epsilon_{j}},\theta_{j},\tau_{j}\bigg)d\h=\int_{Q_{1}^{t}(\frac{\xo}{\rho_j})\cap\gamma_{j}}\psi
\bigg(\frac{\rho_j}{\epsilon_{j}}x,\theta_{j}(\rho_jx),\tau_{j}(\rho_jx)\bigg)d\h.
\end{equation}
Then denoting by $\tilde\mu_j=\theta_{j}(\rho_{j}x)\otimes\tau_{j}(\rho_{j}x)\h\mrestr\frac{\gamma}{\rho_{j}}$
which, in view of \eqref{eqa4}, satisfies
$\widetilde{\mu}_{j}\cstar\mline$,
we have
\begin{displaymath}
g(\xo)=\lim_{j\to\infty}F_{\frac{\rho_j}{\epsilon_{j}}}\left(\tilde\mu_j, Q_{1}^{t}\left(\displaystyle{\frac{\xo}{\rho_j}}\right)\right).
\end{displaymath}
Now applying Proposition \ref{prop-flat-limit} we conclude that
$$
g(x_0)\geq \psi_{\mathrm{hom}}(b,t).
$$
Therefore for $\h$-a.e. $\xo\in \Omega\cap\gamma$ one has
\begin{equation}\label{eq7}
\frac{d\nu}{d\h\mrestr\gamma}(\xo)\geq\psi_{\mathrm{hom}}(\theta(x_0),\tau(x_0)).
\end{equation}

{\em Step 3. (conclusion)} The liminf inequality is achieved by integrating over $\Omega\cap\gamma$. In fact from $(\ref{eqo})$ and $(\ref{eq7})$ we get
\begin{displaymath}
\nu(\Omega)\geq\int_{\Omega\cap\gamma}\frac{d\nu}{d\h\mrestr\gamma} d\h\geq\int_{\Omega\cap\gamma}\psi_{\mathrm{hom}}(\theta,\tau)d\h.
\end{displaymath}
Since $\nu_{j}\cstar\nu$, it follows that $\liminf_{j\to\infty}\nu_{j}(\Omega)\geq\nu(\Omega)$ and so
\begin{align*}
\liminf_{j\to\infty}F_{\epsilon_{j}}(\mu_{j})=\liminf_{j\to\infty}\nu_{j}(\Omega)&\geq\nu(\Omega)\\&\geq\int_{\Omega\cap\gamma}\psi_{\mathrm{hom}}(\theta,\tau)d\h = F_{\mathrm{hom}}(\mu)
\end{align*}
as desired.

\subsection{The limsup inequality}\label{upper}
We complete the proof by exhibiting the construction of a recovery sequence, i.e., given a target measure $\mu\in\M(\Omega)$ we have to find a sequence $\mu_{\epsilon_{j}}\in\M(\Omega)$ such that $\mu_{\epsilon_{j}}\cstar\mu$ and
\begin{displaymath}
\limsup_{j\to\infty}F_{\epsilon_{j}}(\mu_{\epsilon_{j}})\leq F_{\mathrm{hom}}(\mu)
\end{displaymath}
for all $\epsilon_{j}\to 0^{+}$. Using a standard diagonal argument it suffices to show the construction for a dense family. Here we consider the set of measures in $\M(\Omega)$ which are supported on a polyhedral curve $\gamma$. This density result is a consequence of the corresponding result for currents (see Appendix).

{\em Step 1: (polyhedral measures)} Now let $\mu=\sum_{i=1}^{N}b_{i}\otimes t_{i}\h\mrestr\gamma_{i}\in\M(\R^{n})$ be a polyhedral measure, in the sense that the $\gamma_{i}$ are disjoint segments (up to the endpoints), $b_{i}\in\Z^{m}$, $t_{i}\in\Sa^{n-1}$ for $i=1,\ldots,N$. Let $\gamma=\cup_{i=1}^{N}\gamma_{i}$. We cover $\gamma\cap\Omega$, up to a $\h$-null set, with $N$ families of countably many disjoint cubes $\{Q^{k}_i=Q^{t_i}_{r_{k,i}}(x_{k,i})\}_{k\in\N}$ and $i=1,\ldots,N$ which are contained in $\Omega$ and have the property that $\gamma\cap Q^{k}_i=\gamma_i\cap Q^{k}_i$ with $\gamma_i$ though the centre of the cube  so that
\begin{displaymath}
\mu\mrestr Q^{k}_i=b_{i}\otimes t_{i}\h\LL(x_{k,i}+t_i\R)
\end{displaymath}
for $i\in\{1,\ldots ,N\}$. 
In particular $\sum_k r_{k,i}=\h(\gamma_i)$ for every $i=1,\ldots,N$.

By Proposition \ref{prop-flat-limit} and Remark \ref{upper-rem}, for all $k\in\N$ and $i\in\{1,\ldots ,N\}$, there exists a sequence $\mu_{\epsilon_{j}}^{k,i}\in\M(Q^{k}_i)$ such that $\mu_{\epsilon_{j}}^{k,i}\cstar\mu\mrestr Q^{k}_i$, $\su(\mu_{\epsilon_{j}}^{k,i}-b_{i}\otimes t_{i}\h\mrestr(x_{k,i}+\R t_{i}))\subset Q^{k}_i$ and
\begin{equation}\label{dalpasso1}
\lim_j F_{\epsilon_{j}}(\mu_{\epsilon_{j}}^{k,i},Q^{k}_i)\leq r_{k,i}\psi_{\mathrm{hom}}(b_{i},t_{i}).
\end{equation}
Finally define $\mu_{\epsilon_{j}}=\sum_{i=1}^N\sum_{k\in\N}\mu_{\epsilon_{j}}^{k,i}$. By the properties of $\mu_{j}^{k,i}$ we have that $\mu_{\epsilon_{j}}\in\M(\Omega)$, $\mu_{\epsilon_{j}}\cstar\mu$ and
\begin{displaymath}
\lim_j F_{\epsilon_{j}}(\mu_{\epsilon_{j}})\leq F_{\mathrm{hom}}(\mu)
\end{displaymath}
which concludes the proof for $\mu$ polyhedral.

 {\em Step 2: (general measures)} Finally, we deduce the inequality for all measure $\mu\in\M(\Omega)$. We first extend $\mu$ to a measure  $\mathcal{E}\mu\in\M(\R^{n})$, with
\begin{displaymath}
\lim_{\delta\to 0}\lvert\mathcal{E}\mu\lvert(\Omega_{\delta}\setminus\Omega)=0
\end{displaymath}
where $\Omega_{\delta}=\{x:\mathrm{dist}(x,\Omega)<\delta\}$ (see  the Appendix). By Theorem \ref{densitapolichiuse} 
there exist a sequence of polyhedral measures $\mu_{k}\in\M(\R^{n})$ and a sequence of $C^{1}$ and bi-Lipschitz maps $f_{k}$ such that
\begin{displaymath}
\lvert\mu_{k}-(f_{k})_{\sharp}\mathcal{E}\mu\lvert(\R^{n})\to 0,\,\,\|f_{k}-x\|_{L^{\infty}}\to 0,\,\,\|Df_{k}-\mathrm{Id}\|_{L^{\infty}}\to 0.
\end{displaymath}
This implies $\mu_{k}\cstar\mathcal{E}\mu$. By Lemma \ref{continuitasottodef}, the continuity of $\psi_{\mathrm{hom}}$ in $t\in\Sa^{n-1}$ and the invariance under deformations of the multiplicity map one obtains 
\begin{align}\label{jjjj}
F_{\mathrm{hom}}(\mu_{k},\Omega)&\leq F_{\mathrm{hom}}((f_{k})_{\sharp}\mathcal{E}\mu,\Omega)+c\|\mu_{k}-(f_{k})_{\sharp}\mathcal{E}\mu\|\nonumber\\&\leq F_{\mathrm{hom}}(\mathcal{E}\mu_{k},\Omega_{\delta_{k}})(1+c\|Df_{k}-\mathrm{Id}\|_{L^{\infty}})+c\|\mu_{k}-(f_{k})_{\sharp}\mathcal{E}\mu\|
\end{align}
where $\delta_{k}=\|f_{k}-x\|_{L^{\infty}}\to 0$. Taking the limit in (\ref{jjjj}) we get
\begin{equation}\label{kkkk}
\limsup_{k\to\infty}F_{\mathrm{hom}}(\mu_{k},\Omega)\leq F_{\mathrm{hom}}(\mu,\Omega)=F_{\mathrm{hom}}(\mu).
\end{equation}
It then follows from the lower semicontinuity with respect to the weak$^{\ast}$ convergence of $\Gamma\textrm{-}\limsup_{j}F_{\epsilon_{j}}$, the definition of $\Gamma\textrm{-}\limsup$, Step 1 and  (\ref{kkkk}) that
\begin{align*}
\Gamma\hy{-}\limsup_{j}F_{\epsilon_{j}}(\mu)&\leq \liminf_{k}(\Gamma\hy{-}\limsup_{j}F_{\epsilon_{j}}(\mu_{k}))\\&\leq\liminf_{k}F_{\mathrm{hom}}(\mu_{k})\leq\limsup_{k}F_{\mathrm{hom}}(\mu_{k})\leq F_{\mathrm{hom}}(\mu)
\end{align*}
as desired.

\section*{Appendix: Some results for rectifiable currents}

For convenience of the reader here we give the basic definitions and properties for currents in the form that is used in the paper.

\smallskip
{\em Mass of a current:} The total variation of the rectifiable current in (\ref{def1current}) is the measure $\| T\|=\lvert\theta\lvert\h\mrestr\gamma$, its mass is
\begin{displaymath}
\Mb(T)=\| T\|(\Omega)=\int_{\gamma}\lvert\theta\lvert d\h,
\end{displaymath}
and it gives the weighted length of the current $T$ with respect to the Euclidean norm $\lvert	\cdot\lvert$ on $\Z^{m}$. Moreover for any open subset $W\subset\Omega$ we denote $\Mb_{W}(T)=\| T\|(W)$ and we can define the support of the current $T$ as $\mathrm{supp}\| T\|$

\smallskip
{\em Flat norm:} The flat norm of the $k$-current $T$  is defined as follows: for all $W\sbb\Omega$ let
\begin{align*}
\Fl_{W}(T):=\inf\{\Mb_{W}(A)+&\Mb_{W}(B): T=A+\partial B\\& A\in\mathcal{R}_{k}(\Omega,\Z^{m}),\, B\in\mathcal{R}_{k+1}(\Omega,\Z^{m})\}.
\end{align*}

For example the flat norm of a $0$-current given by to point $x_1$ and $x_2$ with multiplicity $+1$ and $-1$ respectively is the length of the segment connecting $x_1$ and $x_2$.

\begin{theorem}[Theorem 31.2, \cite{Sim}]\label{thm31.2} Let $T_{j},T\in \rz$ be such that 
	\begin{displaymath}
	\sup_{j\geq 1}(\Mb_{W}(T_{j})+\Mb_{W}(\partial T_{j}))<\infty
	\end{displaymath}
	for all $W\sbb \Omega$. Then $T_{j}\cstar T$ if and only if $\mathbf{F}_{W}(T_{j}-T)\to 0$ for every $W\sbb \Omega$.
\end{theorem}

\smallskip
{\em Slicing of $1$-currents:} 
 For a current $T\in\rz$ such that $\Mb_{W}(T)+\Mb_{W}(\partial T)<\infty$ for all $W\sbb\Omega$ and a Lipschitz function $f:\R^{n}\rightarrow\R$, one can define the slice of $T$ through $f$ in $t\in\R$ as
\begin{displaymath}
T[f,t-]=\partial(T\mrestr\{x:f(x)<t\})-(\partial T)\mrestr\{x:f(x)<t\}
\end{displaymath}
\begin{displaymath}
T[f,t+]=(\partial T)\mrestr\{x:f(x)>t\}-\partial(T\mrestr\{x:f(x)>t\}).
\end{displaymath}
Up to at most a countable set of point $t\in \R$ for which $\Mb(T\mrestr\{x:f(x)=t\})+\Mb((\partial T)\mrestr\{x:f(x)=t\})>0$, it holds
\begin{displaymath}
T[f,t-]=T[f,t+]=:T[f,t].
\end{displaymath}

These are the main properties of the slicing that we need in the paper (see Section 4.2.1  in \cite{Fed}):
	\begin{itemize}
		\item[$\mathrm{(i)}$] $\su T[f,t+]\subset f^{-1}\{t\}\cap\su T$;
		\item[$\mathrm{(ii)}$] it holds
		\begin{displaymath}
		\int_{a}^{b}\Mb ( T[f,t+])\,dt\leq\mathrm{Lip}(f)\lt\{x:a<f(x)<b\}.
		\end{displaymath}
		for all $-\infty\leq a<b\leq\infty$; 
		\item[$\mathrm{(iii)}$] $T[f,t+]$ is a 0-current for a.e. $t\in\R$;
		\item[$\mathrm{(iv)}$] if $\mathrm{supp}T\subset K$, $K\subset\R^{n}$ a compact set, then
		\begin{displaymath}
		\int\Fl_{K\cap\{x:f(x)=t\}}( T[f,t+])\,dt\leq\mathrm{Lip}(f)\Fl_{K}(T).
		\end{displaymath}
	\end{itemize}

{\em Push forward of $1$-currents:} For a bi-lipschitz map $f:\R^{n}\rightarrow\R^{n}$, the push-forward $f_{\sharp}T$ of $T$ is the current
\begin{equation}\label{pushingf}
\langle f_{\sharp}T,\phi\rangle=\int_{f(\gamma)}\theta(f^{-1}(y))\langle\phi(y);\tau'(y)\rangle d\h(y),
\end{equation}
where $\tau'$ it the tangent to $f(\gamma)$ with the same orientation of $\tau$, $\tau'(f(x))=D_{\tau}f(x)/\lvert D_{\tau}f(x)\lvert$ and $D_{\tau}f(x)$ denotes the tangential derivative of $f$ along $\gamma$, which exists $\h$-a.e. on $\gamma$ since $f$ is Lipschitz on $\gamma$; if $f$ is differentiable in $x$ then $D_{\tau}f(x)=Df(x)\tau(x)$.

\bigskip

 Finally  we state some additional results on which the proof of the liminf inequality also relies. Although both are known results in the theory of scalar currents \cite[ Subsection 4.2.24, Theorems 4.2.16]{Fed}, we refer to \cite[Theorem 2.4, Theorem 2.5]{CGM} for their $\Z^{m}$-valued version and proof. The first one is a compactness result in the class of 1-rectifiable currents without boundary, that, in the liminf, ensures that the limit measure belongs to the space $\M(\Omega)$. 
\begin{theorem}[Compactness]\label{compactness} Let $(T_{j})_{j\in\N}$ be a sequence of rectifiable 1-currents without boundary in $\mathcal{R}_{1}(\R^{n},\Z^{m})$. If
	\begin{displaymath}
	\sup_{j\in\N}\Mb(T_{j})<\infty
	\end{displaymath}
	then there are a current $T\in\mathcal{R}_{1}(\R^{n},\Z^{m})$ and a subsequence $(T_{j_{k}})_{k\in\N}$ such that
	\begin{displaymath}
	T_{j_{k}}\cstar T.
	\end{displaymath}
\end{theorem}
The second theorem gives a characterization of the support of a closed 1-rectifiable current.\begin{theorem}[Structure]\label{struttura} Let $T\in\mathcal{R}_{1}(\R^{n};\Z^{m})$ with $\partial T=0$ and $\Mb(T)<\infty$. Then there are countably many oriented Lipschitz curves $\gamma_{i}$ with tangent vector fields $\tau_{i}:\gamma_{i}\rightarrow\Sa^{n-1}$ and multiplicities $\theta_{i}\in\Z^{m}$ such that
	\begin{displaymath}
	\langle T,\phi\rangle=\sum_{i\in\N}\theta_{i}\int_{\gamma_{i}}\langle\phi(x);\tau_{i}(x)\rangle d\h(x)\quad\forall\phi\in C_{c}^{\infty}(\R^{n},\R^{n}).
	\end{displaymath}
	Further,
	\begin{displaymath}
	\sum_{i}\lvert\theta_{i}\lvert\h(\gamma_{i})\leq\sqrt{m}\Mb(T).
	\end{displaymath}
\end{theorem}

Finally we recall the approximation result for currents with polyhedral currents (this is a classical result for scalar currents. For currents with vector value multiplicity the proof can be found in \cite{CGM}, while the corresponding result for $k$-currents con be found in \cite{BCG}).

\begin{theorem}[Density]\label{densitapolichiuse}
	Fix $\epsilon>0$ and consider a $\Z^{m}$-valued closed 1-current $T\in\mathcal{R}_{1}(\Omega,\Z^{m})$. Then there exist a bijective map $f\in C^{1}(\R^{n};\R^{n})$, with inverse also $C^{1}$, and a closed polyhedral 1-current $P\in\mathcal{P}_{1}(\R^{n},\Z^{m})$ such that
	\begin{displaymath}
	\Mb(f_{\sharp}T-P)\leq\epsilon
	\end{displaymath}
	and
	\begin{displaymath}
	\lvert Df(x)-\mathrm{Id}\lvert +\lvert f(x)-x\lvert\leq\epsilon\quad\forall x\in\R^{n}.
	\end{displaymath}
	Moreover, $f(x)=x$ whenever $\mathrm{dist}(x,\su T)\geq\epsilon$.
\end{theorem} 

 It is important to notice that the deformation result given above guarantees  a current $T$ without boundary can be approximated by polyhedral currents without boundary, or, equivalently, divergence-free measures. In particular one should note that the multiplicity map is invariant under deformation through a bi-Lipschitz function.

The theorem is given on $\R^{n}$ but a local version can be deduced using the extension lemma recalled below \cite[ Lemma 2.3]{CGM}.

\begin{lemma}[Extension]\label{estensione}
	Let $\Omega\subset\R^{n}$ be a  bounded Lipschitz open set. For every closed rectifiable 1-current defined in $\Omega$, $T\in\rz$, there exists a closed rectifiable 1-current $\mathcal{E}T\in\mathcal{R}_{1}(\R^{n};\Z^{m})$ with $\mathcal{E}T\mrestr\Omega=T$ and $\Mb(\mathcal{E}T)\leq c\Mb(T)$. The constant $c$ depends only on $\Omega$. Moreover, $\lim_{\delta\to 0}\Mb(\mathcal{E}T\mrestr(\Omega_{\delta}\setminus\Omega))=0$, where $\Omega_{\delta}=\{x:\mathrm{dist}(x,\Omega)<\delta\}$.
\end{lemma}
Finally we include the following lemma \cite[Lemma 3.3]{CGM}, again useful in proving the limsup inequality.
\begin{lemma}\label{continuitasottodef} Assume that $\psi:\Z^{m}\times\Sa^{n-1}\rightarrow[0,\infty)$ is Borel measurable, obeys $\psi(0,t)=0$, $\psi(b,t)\geq c\lvert b\lvert$ and 
	\begin{displaymath}
	\lvert \bp(b,t)-\bp(b',t')\lvert\leq c\lvert b-b'\lvert + c\lvert b\lvert\lvert t-t'\lvert.
	\end{displaymath}
	Let $\mu,\mu'\in\M(\Omega)$. Then for any open set $\omega\subset\Omega$ we have
	\begin{displaymath}
	\lvert E(\mu,\omega)-E(\mu',\omega)\lvert\leq c\lvert\mu-\mu'\lvert(\omega).
	\end{displaymath}
	Further, if $f:\R^{n}\rightarrow\R^{n}$ is bi-Lipschitz then for any open set $\omega\subset\R^{n}$
	\begin{equation}\label{deformazionefsharp}
	\lvert E(\mu,\omega)-E(f_{\sharp}\mu,f(\omega))\lvert\leq cE(\mu,\omega)\Vert Df-\mathrm{Id}\Vert_{L^{\infty}}. 
	\end{equation}
\end{lemma}

\end{document}